\documentclass[11pt,a4paper]{article}

\usepackage[utf8]{inputenc}
\usepackage[T1]{fontenc}
\usepackage{amsmath,amssymb,amsthm}
\usepackage{mathtools}
\usepackage{booktabs}
\usepackage{graphicx}
\usepackage{xcolor}
\usepackage{hyperref}
\usepackage[margin=1in]{geometry}
\usepackage{enumitem}
\usepackage[numbers]{natbib}

\hypersetup{colorlinks=true, linkcolor=blue!60!black, citecolor=blue!60!black, urlcolor=blue!60!black}

\newtheorem{theorem}{Theorem}
\newtheorem{proposition}{Proposition}
\newtheorem{lemma}{Lemma}
\newtheorem{corollary}{Corollary}
\theoremstyle{definition}
\newtheorem{remark}{Remark}

\title{\textbf{Static Equilibria of Perturbed Spheres}\\[6pt]
\large A Single-Harmonic Class Map, a Parity Obstruction,\\
and a Certified Counter for the Mono-Monostatic Regime}
\author{Vincent Wesley Couey\\[2pt]\small Independent researcher\\[2pt]
\small\texttt{vinnycouey@gmail.com}\\[2pt]
\small ORCID: \href{https://orcid.org/0009-0005-6869-308X}{0009-0005-6869-308X}}
\date{}

\begin{document}
\maketitle

\begin{abstract}
V\'arkonyi and Domokos proved that homogeneous convex bodies with any prescribed numbers $S\ge1$ of stable and
$U\ge1$ of unstable static equilibria exist, the case $S=U=1$ being the mono-monostatic G\"omb\"oc. Their
result is an existence statement. We give a complete \emph{constructive} answer for the simplest
nondegenerate shapes: a homogeneous body whose boundary is the unit sphere perturbed radially by a single real
spherical harmonic $Y_\ell^m$ of degree $\ell\ge2$ and order $1\le m\le\ell$ has exactly
\[
S \;=\; U \;=\; m(\ell-m+1)
\]
stable and unstable equilibria, for every amplitude in the convex range when $m\ge2$, and for small amplitude
when $m=1$. The reduction to the critical points of the perturbing harmonic is in fact an identity when
$m\ge2$: the symmetry of a single tesseral harmonic pins the centroid at the origin exactly, so the
centroid-to-surface distance is a strictly increasing function of the harmonic and the two problems coincide
with no error term. We then count the critical points of $Y_\ell^m$ from its separable structure and verify
the Poincar\'e--Hopf balance in index form, the polar monkey-saddles persisting unsplit with index $1-m$.
Three consequences follow. (i) A single harmonic populates only the \emph{diagonal}
$S=U$ of the equilibrium-class lattice; in particular no single harmonic of degree $\ge2$ is
mono-monostatic. (ii) Any centrally symmetric (even-degree) perturbation has even $S$ and $U$ and cannot be
mono-monostatic, a parity obstruction. (iii) Mono-monostaticity is therefore intrinsically multi-harmonic.
A predict-then-confirm numerical study on eight single-harmonic bodies matches the formula exactly. For the
complementary multi-harmonic regime, where mono-monostatic bodies live, we give a \emph{certified} equilibrium
counter (interval arithmetic on the centroid-to-surface distance, per-box Krawczyk uniqueness,
interval-Hessian classification, and stereographic polar charts) that, given the centroid, provably neither
under- nor over-counts, validated against the single-harmonic law and the polar critical structure. It
certifies that a member of a known analytic family, at an amplitude of our own choosing inside the strictly
convex sub-regime rather than at the amplitude discussed there, is mono-monostatic. For that body the
centroid is exactly the origin, so its certificate is unconditional. This resolves by certified computation a case that drainage-basin and seed-based methods
leave ambiguous.

\end{abstract}

% =====================================================================
\section{Introduction}

A homogeneous convex body resting on a horizontal plane under gravity sits in equilibrium when its centroid
lies vertically above the contact point. How many such equilibria a body admits, and of which stability types,
is a classical question brought back into focus by Arnold's 1995 conjecture and its resolution by V\'arkonyi
and
Domokos~\cite{varkonyi2006a,varkonyi2006b}: there exist homogeneous convex bodies with exactly one stable and
one unstable equilibrium (mono-monostatic bodies, the G\"omb\"oc), and more generally bodies realizing any
class $(S,U)$ with $S,U\ge1$. Writing $S$, $U$, $H$ for the numbers of stable, unstable and saddle equilibria,
the Poincar\'e--Hopf theorem on $S^2$ gives $S-H+U=2$ whenever every equilibrium is
nondegenerate~\cite{varkonyi2006b}. In general only the index form is available, the sum of the
Poincar\'e--Hopf indices being $\chi(S^2)=2$; the distinction is not academic here, and we return to it in
Remark~\ref{rem:countform}.

The V\'arkonyi--Domokos theorem is an \emph{existence} result: it guarantees a body for each class but does
not, in elementary closed form, say which class a \emph{given} shape realizes. The mono-monostatic bodies it
produces are famously close to a sphere: numerical analysis there indicates a shape parameter below
$5\times10^{-5}$ is needed to meet convexity together with the remaining
constraints~\cite{varkonyi2006b}, and their explicit geometry is delicate; for two analytic
parameterizations see~\cite{sloan2023}. This motivates a
constructive question at the opposite, simplest end of the problem:

\begin{quote}
\emph{What is the equilibrium class $(S,U)$ of the simplest nondegenerate perturbation of a sphere, namely a
single spherical-harmonic radial deformation?}
\end{quote}

We answer it completely for $m\ge2$, and for $m=1$ below an amplitude threshold. Our main result
(Theorem~\ref{thm:main}) is that a single real harmonic $Y_\ell^m$ ($\ell\ge2$, $m\ge1$) produces exactly
$S=U=m(\ell-m+1)$ equilibria of each stability type. Several ingredients are classical and we use them as
such: equilibria are the critical points of the centroid-to-surface distance~\cite{varkonyi2006b}; the
critical points of a single tesseral harmonic follow from its separable structure and the zeros of
associated Legendre functions~\cite{hobson,szego}; and the parity facts below are older than they look, as
we note in Section~\ref{sec:consequences}. What is new is the following.
\emph{(i)} We show the passage from the body to the harmonic is not merely a leading-order approximation but,
for every order $m\ge2$, an \emph{exact identity}: the symmetry of a single tesseral harmonic forces the
perturbed centroid to the origin exactly (Lemma~\ref{lem:centroid}), so the count holds for the entire convex
amplitude range rather than only asymptotically (Lemma~\ref{lem:reduction}). \emph{(ii)} Assembling this into
the closed-form class map $Y_\ell^m\mapsto(m(\ell-m+1),m(\ell-m+1))$, valid at every degree and order, and
reading off its structural consequences (Section~\ref{sec:consequences}): the reachable classes are exactly
the diagonal $\{(n,n):n\ge2\}$, and no single harmonic is mono-monostatic. \emph{(iii)} For the complementary
multi-harmonic regime we give a certified counter (below).

Two adjacent literatures should be distinguished from this one, and one of them overlaps it. First, the
critical points of \emph{random} spherical harmonics~\cite{nicolaescu,cammarota} grow like $\ell^2$; a single
tesseral harmonic is far below that, and its count is exactly $m(\ell-m+1)$. Second, and more closely, the
degree-three case is classical under another name: a traceless symmetric rank-three tensor restricted to
$S^2$ is the \emph{octupolar potential}, whose critical points were classified by Gaeta and
Virga~\cite{gaetavirga2016,gaetavirga2019,gvreview} and by Walcher~\cite{walcher2019}. Our $\ell=3$ rows
reproduce their classification rather than extend it: $Y_3^1\mapsto(3,3,4)$ is their generic count and
$Y_3^3$ their minimal one, and $Y_3^2$ attains $14=2(d^2-d+1)$ critical points, the maximum shown to be
realizable within the traceless class~\cite{kozhasov,kuznetsov1992}. The contribution here is the
closed-form $(\ell,m)\mapsto(S,U)$ map at \emph{every} degree, together with the exactness of
Lemma~\ref{lem:reduction}; for $\ell=3$ our value is an independent, exact-arithmetic confirmation of a
published count.

The single-harmonic classification shows that mono-monostatic bodies must be multi-harmonic
(Corollaries~\ref{cor:nomono}--\ref{cor:parity}); these are near-spherical, and counting their equilibria
reliably is delicate. In our own experience on these bodies, a seed-based critical-point finder undercounts
(small basins of attraction are missed) while a grid-based basin method overcounts (the count fails to
converge under refinement); the drainage-basin behaviour reported in~\cite{couey2026mono} is an instance. So
the mono-monostatic status of a given near-spherical shape can be genuinely ambiguous in practice. Our second
contribution (Section~\ref{sec:certified}) is a \emph{certified} equilibrium counter that replaces this
ambiguity with a decision: where it terminates it returns a count that provably neither over- nor
under-counts, and where it does not it returns undecided boxes rather than a number. It encloses every
critical point of the centroid-to-surface distance by interval arithmetic,
certifies each by the Krawczyk test, classifies it by its interval Hessian, and treats the poles in
stereographic charts. Validated against the single-harmonic law, it certifies, for example, that a known
analytic parameterization of the sphere~\cite{sloan2023} yields a mono-monostatic body.

% =====================================================================
\section{Setup and the reduction lemma}\label{sec:setup}

Let $g:S^2\to\mathbb{R}$ be smooth and let $B_\varepsilon$ be the body with radial boundary
\begin{equation}
r(u) = 1 + \varepsilon\, g(u), \qquad u\in S^2,
\label{eq:body}
\end{equation}
taken homogeneous and, for $\varepsilon$ small enough, convex. Its boundary point in direction $u$ is
$x(u)=r(u)\,u$ and its centroid is $\mathbf{c}=\mathbf{c}(\varepsilon)$. A direction $u$ is a static
equilibrium iff the surface normal at $x(u)$ passes through $\mathbf{c}$, equivalently iff $u$ is a critical
point of the squared centroid-to-surface distance
\begin{equation}
\rho(u) \;=\; \lvert x(u)-\mathbf{c}\rvert^2,
\label{eq:rho}
\end{equation}
with local minima of $\rho$ stable, local maxima unstable, and saddles
saddle-type~\cite{varkonyi2006b}. (Intuitively the body rests on the surface point nearest the centroid.)

\begin{lemma}[Centroid]\label{lem:centroid}
Let $g$ be a spherical harmonic of degree $\ell\ge2$. Then $\mathbf{c}=O(\varepsilon^2)$. If moreover
$g=Y_\ell^m$ with $m\ge2$, then $\mathbf{c}=\mathbf{0}$ \emph{exactly}, for every $\varepsilon$ for which
$B_\varepsilon$ is a body.
\end{lemma}
\begin{proof}
The first moment of the solid is $\int_{B_\varepsilon}x\,dV=\tfrac14\int_{S^2}u\,r(u)^4\,du$. Expanding
$r^4=1+4\varepsilon g+O(\varepsilon^2)$ gives
$\tfrac14\!\int u\,du + \varepsilon\!\int u\,g(u)\,du + O(\varepsilon^2)$. The first integral vanishes by
oddness. The components of $u$ are the degree-one harmonics $Y_1^{m}$; since $g$ has degree $\ell\ge2$,
orthogonality of spherical harmonics of distinct degree gives $\int_{S^2}u\,g(u)\,du=0$. The volume is
$\tfrac13\int r^3=\tfrac{4\pi}{3}+O(\varepsilon^2)$ (the $O(\varepsilon)$ term vanishes as $g\perp Y_0^0$).
Hence $\mathbf{c}=O(\varepsilon^2)$.

For $g=Y_\ell^m$ with $m\ge2$ the first moment vanishes identically, to all orders in $\varepsilon$. Write
$r^4=\sum_{k=0}^{4}\binom{4}{k}\varepsilon^k g^k$; it suffices that $\int_{S^2}u\,g^k\,du=0$ for
$k=0,\dots,4$. \emph{Horizontally:} $g^k=P_\ell^m(\cos\theta)^k\cos^k(m\phi)$ has azimuthal content only in
the modes $\cos(jm\phi)$, $j=0,\dots,k$, whereas $u_x$ and $u_y$ carry only the mode $j m=1$; since $m\ge2$
no $jm$ equals $1$, and the $\phi$-integral annihilates every term, so
$\int u_x g^k\,du=\int u_y g^k\,du=0$. \emph{Vertically:} $u_z=\cos\theta$ is $\phi$-independent, so only the
constant ($j=0$) azimuthal part of $g^k$ survives the $\phi$-integral. That part vanishes for odd $k$, since
$\int_0^{2\pi}\cos^k(m\phi)\,d\phi=0$; and for even $k$ the surviving $\theta$-integrand is
$\cos\theta\,P_\ell^m(\cos\theta)^k$, which is odd in $\cos\theta$ because $P_\ell^m(-x)=(-1)^{\ell+m}
P_\ell^m(x)$ renders $(P_\ell^m)^k$ even for even $k$, and it therefore integrates to zero on $[-1,1]$.
Hence $\mathbf{c}=\mathbf{0}$.
\end{proof}

\begin{remark}\label{rem:m1centroid}
The exact vanishing is special to $m\ge2$. For $m=1$ the mode $jm=1$ is present, the horizontal argument
fails at $k=3$, and $\mathbf{c}$ is in general nonzero: it is $O(\varepsilon^3)$, and vanishes identically
only when $\ell+m$ is odd. The $O(\varepsilon^2)$ bound of the first paragraph is thus never sharp for a
single harmonic; we retain it because it is all that the general degree-$\ell$ statement supports.
\end{remark}

\begin{remark}
The hypothesis $\ell\ge2$ is essential: a degree-one component of $g$ shifts the centroid at order
$\varepsilon$, and such shapes fall outside the present analysis. This is exactly the regime of the explicit
G\"omb\"oc constructions, to which the single-harmonic theorem does not apply; they are treated by the
certified counter of Section~\ref{sec:certified}.
\end{remark}

\begin{lemma}[Reduction]\label{lem:reduction}
Let $g=Y_\ell^m$ with $\ell\ge2$.
\begin{enumerate}[nosep,label=(\roman*)]
\item If $m\ge2$ then $\rho(u)=\big(1+\varepsilon g(u)\big)^2$ \emph{exactly}, and for every $\varepsilon$ in
the convex range the equilibria of $B_\varepsilon$ are in bijection with the critical points of $g$,
preserving both type and index at every critical point, the poles included.
\item If $m=1$ then $\rho(u)=1+2\varepsilon g(u)+O(\varepsilon^2)$, and the same bijection holds for
$\varepsilon$ below a threshold.
\end{enumerate}
In both cases minima of $g$ correspond to stable and maxima to unstable equilibria.
\end{lemma}
\begin{proof}
(i) We may assume $\varepsilon>0$: replacing $\varepsilon$ by $-\varepsilon$ replaces $g$ by $-g$, which for
$g=Y_\ell^m$ is the rotation $\phi\mapsto\phi+\pi/m$, so the two bodies are congruent and the counts agree.
By Lemma~\ref{lem:centroid}, $\mathbf{c}=\mathbf{0}$ when $m\ge2$, so
$\rho=\lvert r(u)\,u\rvert^2=r(u)^2=(1+\varepsilon g(u))^2$ identically, with no expansion and no remainder.
In the convex range $r=1+\varepsilon g>0$, so $\rho=\Phi\circ g$ with $\Phi(s)=(1+\varepsilon s)^2$ and
$\Phi'(s)=2\varepsilon(1+\varepsilon s)>0$ there, i.e. $\Phi$ strictly increasing on the relevant range. Hence $\nabla\rho=2\varepsilon r\,\nabla g$ vanishes exactly where $\nabla g$ does, and
composition with a strictly increasing diffeomorphism preserves the local type and the Poincar\'e--Hopf index
of every critical point. The correspondence is therefore exact, with no smallness threshold beyond convexity
itself.

(ii) For $m=1$, from \eqref{eq:body}--\eqref{eq:rho} and $\lvert u\rvert=1$,
$\rho=\lvert u+\varepsilon g u-\mathbf{c}\rvert^2=1+2\varepsilon g(u)-2\,u\!\cdot\!\mathbf{c}+O(\varepsilon^2)$,
and $2u\!\cdot\!\mathbf{c}=O(\varepsilon^3)$ by Remark~\ref{rem:m1centroid}, so
$\rho=1+2\varepsilon g+O(\varepsilon^2)$. For $m=1$ every critical point of $g$ is nondegenerate, the poles
being regular (Proposition~\ref{prop:count}); nondegenerate critical points persist and keep their type under
the $O(\varepsilon^2)$ perturbation by the implicit function theorem, and the bijection follows for
$\varepsilon$ below a threshold $\varepsilon_0(\ell,1)$.
\end{proof}

\begin{remark}
The equivalence between static equilibria and critical points of the centroid-to-surface distance $\rho$ is
the reduction underlying the V\'arkonyi--Domokos classification~\cite{varkonyi2006b}; we use it unchanged. What
Lemma~\ref{lem:reduction} adds is its collapse onto the harmonic itself. For $m\ge2$ this collapse is not
asymptotic but an identity: the symmetry of a single tesseral harmonic pins the centroid at the origin
exactly, so $\rho$ is a strictly increasing function of $g$ throughout the convex range, and the equilibrium
problem \emph{is} the critical-point problem for $g$, with no error term to control. This is what makes the
count computable in closed form (Section~\ref{sec:count}), and it is the step that fails once $g$ carries a
degree-one component, i.e. for the explicit near-spherical G\"omb\"oc constructions, where the centroid moves
at order $\varepsilon$.
\end{remark}

% =====================================================================
\section{Critical points of a tesseral harmonic}\label{sec:count}

We use the real harmonic $Y_\ell^m(\theta,\phi)=P_\ell^m(\cos\theta)\cos(m\phi)$ with colatitude
$\theta\in[0,\pi]$, longitude $\phi$, and associated Legendre function $P_\ell^m$. Throughout $1\le m\le\ell$;
the zonal case $m=0$ is axisymmetric, its critical sets are latitude circles, and it is excluded.

Every count below is invariant under rescaling $g\mapsto g/c$ for a constant $c>0$, since this leaves the
critical set and the type of each critical point untouched. The convex range of $\varepsilon$ is not
scale-invariant, however, so wherever a numerical value of $\varepsilon$ appears (Section~\ref{sec:numerics})
we fix the scale by normalizing the perturbation to $\max_{S^2}\lvert g\rvert=1$, which makes $\varepsilon$
the true radial amplitude of the deformation.

\begin{proposition}\label{prop:count}
For $1\le m\le\ell$ the harmonic $Y_\ell^m$ has exactly $m(\ell-m+1)$ local maxima and $m(\ell-m+1)$ local
minima on $S^2$, together with $2m(\ell-m)$ nondegenerate saddles and the two poles, which are critical of
index $1-m$ for $m\ge2$ and regular, non-critical, for $m=1$. For $m=2$ a pole is an ordinary nondegenerate
saddle (index $-1$); for $m\ge3$ it is a degenerate monkey saddle. The Poincar\'e--Hopf balance holds:
\[
2m(\ell-m+1)\;-\;2m(\ell-m)\;+\;2(1-m)\;=\;2=\chi(S^2).
\]
\end{proposition}
\begin{proof}
Write $x=\cos\theta$ and set $f(\theta)=P_\ell^m(\cos\theta)$, so that $Y_\ell^m=f(\theta)\cos(m\phi)$; all
derivatives below are with respect to $\theta$ unless a variable is named.

\emph{The critical latitudes.} The Legendre function $P_\ell^m(x)=(1-x^2)^{m/2}\,u(x)$ with
$u=d^m P_\ell/dx^m$, a polynomial of degree $\ell-m$ with exactly $\ell-m$ simple zeros in $(-1,1)$ (apply
Rolle $m$ times to the $\ell$ simple zeros of $P_\ell$); and $P_\ell^m(\pm1)=0$ for $m\ge1$~\cite{hobson,szego}.
Differentiating,
\[
\frac{d}{dx}P_\ell^m = (1-x^2)^{\frac{m}{2}-1}\,q(x), \qquad q(x):=(1-x^2)u'(x)-m\,x\,u(x),
\]
so on $(-1,1)$ the critical points of $P_\ell^m$ are exactly the zeros of $q$. Now $q$ has degree
\emph{exactly} $\ell-m+1$: the two degree-$(\ell-m+1)$ contributions do not cancel, their leading coefficients
summing to $-\ell$ times that of $u$. Meanwhile $P_\ell^m$ vanishes at the $\ell-m$ interior zeros of $u$ and
at $x=\pm1$, giving $\ell-m+2$ zeros in $[-1,1]$, so Rolle supplies \emph{at least} $\ell-m+1$ zeros of $q$ in
$(-1,1)$. A polynomial of degree $\ell-m+1$ with $\ell-m+1$ zeros there has no others and all are simple.
Hence $f$ has exactly $\ell-m+1$ critical latitudes $\theta_1,\dots,\theta_{\ell-m+1}$, each nondegenerate
($\ddot f\ne0$), one strictly inside each interval between consecutive zeros of $f$ (the poles counted as
zeros), and $f$ alternates in sign from one such interval to the next.

\emph{The interior critical points.} Setting $\partial_\phi Y_\ell^m=-m\,f(\theta)\sin(m\phi)=0$ and
$\partial_\theta Y_\ell^m=\dot f(\theta)\cos(m\phi)=0$ simultaneously requires
$[\,\sin(m\phi)=0$ or $f=0\,]$ and $[\,\dot f=0$ or $\cos(m\phi)=0\,]$. Of the four combinations, $\sin(m\phi)=
\cos(m\phi)=0$ is impossible, and $f=\dot f=0$ is impossible because the zeros of $f$ in $(0,\pi)$ are simple.
The remaining two are:
\begin{enumerate}[nosep]
\item[(A)] $\sin(m\phi)=0$ and $\theta=\theta_j$: the $2m$ longitudes $\phi=k\pi/m$ at each critical latitude.
\item[(B)] $\cos(m\phi)=0$ and $f(\theta)=0$: the $2m$ longitudes $\phi=(2k{+}1)\pi/2m$ at each of
the $\ell-m$ interior zero-latitudes.
\end{enumerate}
\emph{Classification.} The Hessian of $Y_\ell^m$ in the coordinates $(\theta,\phi)$ has entries
\[
\partial_\theta^2 Y=\ddot f\cos(m\phi), \qquad
\partial_\theta\partial_\phi Y=-m\,\dot f\sin(m\phi), \qquad
\partial_\phi^2 Y=-m^2 f\cos(m\phi).
\]
At a critical point the Christoffel terms drop out, so this coordinate array is the covariant Hessian as a
bilinear form. The round metric is positive definite, so away from the poles the local type (minimum,
maximum, saddle) is the signature of that bilinear form and may be read off the display directly; no index
raising is required, and by Sylvester's law of inertia the mixed Hessian $g^{-1}H$ would in any case have the
same signature.

At a type-(B) point $f=0$ and $\cos(m\phi)=0$, so $Y_\ell^m=0$ and \emph{both} diagonal entries vanish. The
Hessian is purely off-diagonal, $\left(\begin{smallmatrix}0&h\\h&0\end{smallmatrix}\right)$ with
$h=-m\,\dot f\sin(m\phi)\ne0$ (the zeros of $f$ are simple, so $\dot f\ne0$, and $\sin(m\phi)=\pm1$), giving
eigenvalues $\pm\lvert h\rvert$. All $2m(\ell-m)$ of these points are therefore saddles.

At a type-(A) point $\sin(m\phi)=0$ and $\dot f=0$, so the off-diagonal entry vanishes and the Hessian is
diagonal, $\operatorname{diag}\big(\ddot f\cos(m\phi),\,-m^2 f\cos(m\phi)\big)$, evaluated at a latitudinal
extremum of $f$. On a positive hump ($f>0$ at a local max, so $\ddot f<0$) the $m$ longitudes with
$\cos(m\phi)=+1$ make both entries negative, a maximum, and the $m$ with $\cos(m\phi)=-1$ make both positive, a
minimum; on a negative hump the assignment reverses. Each of the $\ell-m+1$ humps thus contributes $m$ maxima
and $m$ minima, giving $m(\ell-m+1)$ of each, all nondegenerate since $\ddot f\ne0$ and $f\ne0$ there. Near a pole, in a local conformal chart $z$ centred there, $Y_\ell^m\sim\theta^m\cos m\phi=\mathrm{Re}(z^m)$
up to a positive factor. For $m=1$ this is regular; for $m=2$ it is $\mathrm{Re}(z^2)$, an ordinary
nondegenerate saddle of index $-1$; for $m\ge3$ it is a degenerate monkey saddle. In every case
$\mathrm{Re}(z^m)$ has index $1-m$ at the origin. Summing indices $(+1)$ for maxima and minima, $(-1)$ for the
ordinary saddles, $(1-m)$ for each pole yields $2$.
\end{proof}

% =====================================================================
\section{The equilibrium-class theorem}

\begin{theorem}\label{thm:main}
Let $B_\varepsilon$ be the homogeneous body \eqref{eq:body} with $g=Y_\ell^m$, $\ell\ge2$, $1\le m\le\ell$.
For $m\ge2$ and \emph{every} $\varepsilon$ in the convex range, and for $m=1$ and all sufficiently small
$\varepsilon>0$, $B_\varepsilon$ has exactly
\[
S \;=\; U \;=\; m(\ell-m+1)
\]
stable and unstable equilibria. Its remaining critical points are the $2m(\ell-m)$ ordinary saddles of
Proposition~\ref{prop:count} together with the two poles, which persist \emph{unsplit} as critical points of
index $1-m$ (ordinary saddles for $m=2$, monkey saddles for $m\ge3$, and regular non-critical points for
$m=1$). The Poincar\'e--Hopf balance closes in index form,
\[
\underbrace{2m(\ell-m+1)}_{\text{extrema}}\;-\;\underbrace{2m(\ell-m)}_{\text{ordinary saddles}}\;+\;
\underbrace{2(1-m)}_{\text{poles}}\;=\;2\;=\;\chi(S^2).
\]
\end{theorem}
\begin{proof}
Immediate from Lemma~\ref{lem:reduction} and Proposition~\ref{prop:count}: the equilibria of $B_\varepsilon$
are the critical points of $g=Y_\ell^m$ with their types and indices, stable corresponding to minima
($m(\ell-m+1)$) and unstable to maxima ($m(\ell-m+1)$). For $m\ge2$ Lemma~\ref{lem:reduction}(i) is an exact
identity, so no critical point, polar or otherwise, is created, destroyed or split at any $\varepsilon$ in the
convex range, and the index sum is verbatim that of Proposition~\ref{prop:count}. For $m=1$ every critical
point is nondegenerate and persists by Lemma~\ref{lem:reduction}(ii).
\end{proof}

\begin{remark}\label{rem:nosplit}
The polar critical points are degenerate for $m\ge3$, so it is natural to ask whether they split into ordinary
saddles on the actual body, which would alter $H$. They do not. By Lemma~\ref{lem:centroid} there is no
centroid shift at any order for $m\ge2$, hence $\rho=(1+\varepsilon g)^2$ exactly and there is no perturbation
available to unfold them: the monkey saddles persist for every $\varepsilon$ in the convex range. This is also
what symmetry demands. The body is exactly invariant under $\phi\mapsto\phi+2\pi/m$, a rotation fixing only
the poles and acting freely nearby, so critical points appearing off a pole must arrive in full orbits of size
$m$; a splitting into $m-1$ ordinary saddles is therefore not equivariantly realizable for $m\ge3$. Were the
symmetry broken by an external perturbation, the generic $\mathbb{Z}_m$-equivariant unfolding of
$\mathrm{Re}(z^m)$ would yield one extremum together with a ring of $m$ saddles, of total index $1-m$ as it
must be, which would change $S$ or $U$ by one. It is the exactness of Lemma~\ref{lem:reduction}(i), not a
genericity assumption, that rules this out.
\end{remark}

\begin{remark}[The count form of Poincar\'e--Hopf fails on these bodies]\label{rem:countform}
The relation $S-H+U=2$, standard in the equilibrium-classification literature, is a \emph{Morse} identity:
it presumes every saddle-type equilibrium has index $-1$. The sectoral bodies $g=Y_m^m$ with $m\ge3$
violate that hypothesis, and they are not pathologies constructed for the purpose. They are single
harmonics, strictly convex throughout the amplitude range verified in Section~\ref{sec:numerics}, and by
Theorem~\ref{thm:main} they carry $S=U=m$ extrema, \emph{no} ordinary saddles, and two polar monkey saddles
of index $1-m$. Counting saddle-type equilibria gives $H=2$, so
\[
S-H+U \;=\; 2m-2 \;\neq\; 2 \qquad (m\ge3),
\]
while the index form closes exactly: $2m\cdot(+1)+2\cdot(1-m)=2$. Hence $Y_3^3,Y_4^4,Y_5^5,\dots$ is an
explicit infinite family of strictly convex homogeneous bodies, in closed form, on which the count identity
is false and only the index identity survives. Any equilibrium count that uses $S-H+U=2$ as a closure check
will reject a correct count on such a body; we therefore verify the index form throughout, including in the
certified counter of Section~\ref{sec:certified}.
\end{remark}

% =====================================================================
\section{Three consequences}\label{sec:consequences}

\begin{corollary}[Diagonal reachability]\label{cor:diag}
A single harmonic perturbation realizes only equilibrium classes on the diagonal $S=U=m(\ell-m+1)$ of the
$(S,U)$ lattice. The reachable values are $\{\,m(\ell-m+1): \ell\ge2,\,1\le m\le\ell\,\}$, i.e.
$2,3,4,5,6,\dots$ (every integer $\ge2$ occurs; e.g. $n$ via $Y_n^1$).
\end{corollary}

\begin{figure}[h]
\centering
\includegraphics[width=0.62\textwidth]{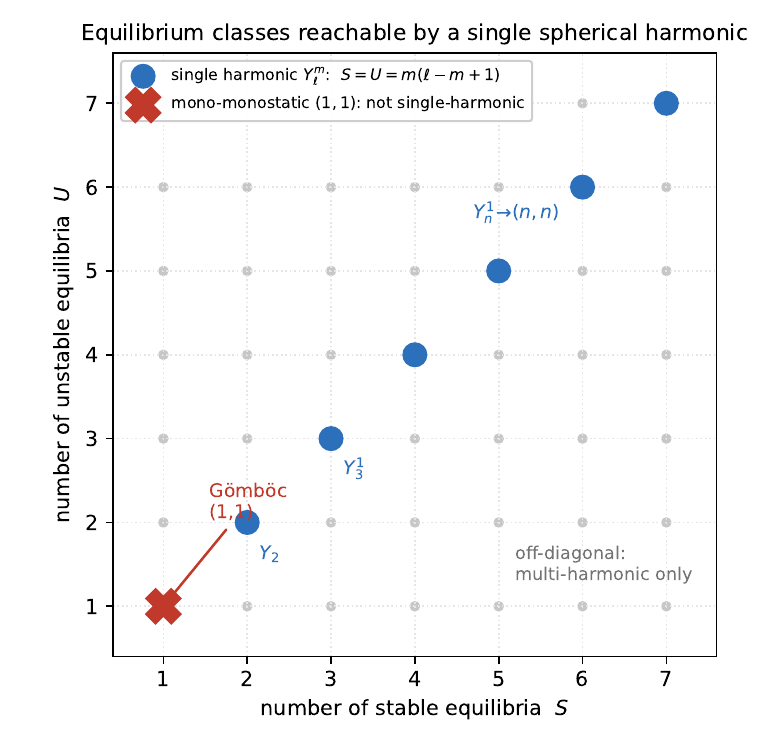}
\caption{The equilibrium-class lattice $(S,U)$. A single spherical harmonic $Y_\ell^m$ realizes only the
diagonal values $S=U=m(\ell-m+1)\in\{2,3,4,\dots\}$ (blue); every integer $\ge2$ occurs, e.g. $Y_n^1$ gives
$(n,n)$. The mono-monostatic corner $(1,1)$ (red) and every off-diagonal class require multi-harmonic
content.}
\label{fig:lattice}
\end{figure}

\begin{corollary}[No single harmonic is mono-monostatic]\label{cor:nomono}
$m(\ell-m+1)=1$ forces $m=1$ and $\ell=1$, a degree-one perturbation (a rigid translation, not a shape
change). Hence no single spherical harmonic of degree $\ell\ge2$ produces a mono-monostatic body.
\end{corollary}

\begin{corollary}[Parity obstruction]\label{cor:parity}
If the perturbation $g$ is a sum of even-degree harmonics, then $g(-u)=g(u)$, the body $B_\varepsilon$ is
centrally symmetric, and its equilibria occur in antipodal pairs. Consequently, \emph{provided the
equilibria are isolated}, $S$ and $U$ are even; and $B_\varepsilon$ cannot be mono-monostatic, with or
without that proviso. Mono-monostaticity therefore requires odd-degree content.
\end{corollary}
\begin{proof}
A spherical harmonic of degree $\ell$ satisfies $Y_\ell^m(-u)=(-1)^\ell Y_\ell^m(u)$, so a sum of even-degree
harmonics is even, whence $r(-u)=r(u)$ and $\mathbf{c}=0$. Then $\rho(-u)=\rho(u)$, so the antipodal map is a
fixed-point-free involution of the equilibrium set preserving stability type; equilibria pair up and $S,U$
are even. A mono-monostatic body has $S=U=1$, odd, a contradiction.
\end{proof}

Corollary~\ref{cor:parity} is stated here for its harmonic reading, but its proof uses neither
Theorem~\ref{thm:main} nor the single-harmonic enumeration: it holds verbatim for \emph{any} centrally
symmetric homogeneous convex body with isolated equilibria, by the same antipodal-involution argument. We keep
the corollary framing because that is how it is used below, while noting the wider scope.

\begin{remark}[Attribution]\label{rem:parityprior}
Neither half of Corollary~\ref{cor:parity} is new, and we state it here only because it is what the rest of
the section uses. The evenness of $S$ and $U$ under central symmetry is due to Kuznetsov and
Kholshevnikov~\cite{kuznetsov1992}, who record, for a degree-$n$ form restricted to the sphere, that for odd
$n$ antisymmetry forces $\nu_0=\nu_2$, while ``for even $n$ it follows from symmetry relative to the center
that $\nu_0$ and $\nu_2$ are even'' ($\nu_0,\nu_2$ being the numbers of minima and maxima). Their paper is
addressed to the geopotential and to our knowledge has not been cited in the equilibrium-classification
literature. Independently, the qualitative exclusion follows from Theorem~2(i)
of~\cite{domokos2023symmetry}: the symmetry group of a mono-monostatic body must fix its unique stable and
its unique unstable point, hence fixes a line, and the antipodal map does not fix either point. For even
ternary forms the non-mono-monostaticity is also implied, without harmonicity, by Maccioni's bound
$t\ge\max(3,2c+1)$~\cite{maccioni}. Our proof of Corollary~\ref{cor:parity} is the short antipodal-involution
argument, which holds verbatim for any centrally symmetric homogeneous convex body with isolated equilibria.
\end{remark}

It is worth recording that near-sphericity is not itself the obstruction. Domokos, L\'angi and V\'arkonyi
show (their Theorem~2(ii)) that mono-monostatic bodies exist within arbitrarily small Hausdorff distance of
the ball~\cite{domokos2023symmetry}, sharpening a construction of
L\'angi~\cite{langi2021}. What
Corollaries~\ref{cor:diag}--\ref{cor:parity} exclude is therefore structural rather than metric: a single
mode cannot reach $(1,1)$ at any amplitude, and no centrally symmetric shape can, however close to the sphere
it is. Mono-monostaticity is in this sense an intrinsically multi-harmonic, odd-parity phenomenon.

% =====================================================================
\section{Numerical confirmation}\label{sec:numerics}

We confirmed Theorem~\ref{thm:main} in two independent ways. (i) \emph{Direct field enumeration}: by
Lemma~\ref{lem:reduction} the equilibrium set is the critical set of $g=Y_\ell^m$ (exactly, for $m\ge2$), so we
counted the local minima, maxima and saddles of $Y_\ell^m$ on a refined icosphere of $S^2$ directions; for a
single harmonic these critical points are well separated and the count is stable under grid refinement. (ii)
\emph{Body solver}: independently, we located equilibria as critical points of $\rho$ in \eqref{eq:rho} on the
body itself by Newton iteration on $S^2$ from a dense seed set, deduplicating by direction and classifying by
sampling $\rho$ on a small geodesic circle. The sectoral prediction $S=U=m$ was \emph{registered before} the
runs; the tesseral cases then revealed the general $m(\ell-m+1)$ law. The two methods agree with
Theorem~\ref{thm:main} and close the Poincar\'e--Hopf index sum at $2$ (for $m\le2$, where every equilibrium
is nondegenerate, this coincides with the count form $S-H+U=2$; for $m\ge3$ it does not, by
Remark~\ref{rem:countform}). As a check that the enumeration is \emph{complete} and not merely
consistent, we also ran an exhaustive interior critical-point search on a fine product grid for a representative
case ($Y_6^2$): it returned exactly $36$ interior critical points, $10$ maxima, $10$ minima and $16$ ordinary
saddles, with no others, matching the interior enumeration of Proposition~\ref{prop:count} term by term
($m(\ell-m+1)=10$ of each extremum type and $2m(\ell-m)=16$ ordinary saddles), the two poles being accounted
separately.

\begin{table}[h]
\centering
\caption{Predicted $m(\ell-m+1)$ versus measured $(S,U)$ for single-harmonic bodies, at the amplitude
$\varepsilon$ actually used (perturbation normalized to $\max\lvert g\rvert=1$). Agreement is exact.
$\varepsilon_0(\ell,m)$ is the strict-convexity threshold computed independently (Section~\ref{sec:numerics});
every body listed sits inside its convex range, with the stated headroom. The high-order sectoral rows
($Y_4^4$, $Y_5^5$) are from the direct field enumeration, which is resolution-robust for these well-separated
critical sets.}
\begin{tabular}{lcccccc}
\toprule
$Y_\ell^m$ & $\varepsilon$ & $\varepsilon_0(\ell,m)$ & headroom & $m(\ell-m+1)$ & measured $S$ & measured $U$ \\
\midrule
$Y_2^2$ & 0.020 & 0.200 & $10.0\times$ & 2 & 2 & 2 \\
$Y_3^1$ & 0.025 & 0.108 & $4.3\times$  & 3 & 3 & 3 \\
$Y_3^2$ & 0.025 & 0.136 & $5.4\times$  & 4 & 4 & 4 \\
$Y_3^3$ & 0.030 & 0.100 & $3.3\times$  & 3 & 3 & 3 \\
$Y_4^2$ & 0.025 & 0.076 & $3.0\times$  & 6 & 6 & 6 \\
$Y_4^4$ & 0.020 & 0.059 & $2.9\times$  & 4 & 4 & 4 \\
$Y_5^5$ & 0.020 & 0.038 & $1.9\times$  & 5 & 5 & 5 \\
$Y_3^1$ & 0.001 & 0.108 & $108\times$  & 3 & 3 & 3 \\
\bottomrule
\end{tabular}
\end{table}

Because the reduction is exact for $m\ge2$ (Lemma~\ref{lem:reduction}(i)), the only precondition these bodies
must satisfy is strict convexity. We verify it rather than assume it: for each row we compute the minimum
Gaussian curvature $K_{\min}$ from the first and second fundamental forms on a dense $(\theta,\phi)$ grid and
locate $\varepsilon_0(\ell,m)$, the largest amplitude with $K_{\min}>0$, by bisection. All eight bodies pass
with at least $1.9\times$ headroom, the tightest being $Y_5^5$. (We use the curvature test, not a
volume-to-hull ratio, which is satisfied by bodies carrying small regions of negative curvature.) The grid
excludes a neighborhood of the poles, where the $(\theta,\phi)$ chart is singular; this hides nothing, since
$P_\ell^m(\pm1)=0$ for $m\ge1$ means the perturbation vanishes at the poles and the boundary is there $C^2$-close
to the unit sphere, hence strictly convex. The
$\varepsilon=10^{-3}$ row confirms that genuinely shallow equilibria are resolved, not missed. We note the
practical limitation, relevant to Section~\ref{sec:certified}, that the body solver converges reliably only when
equilibria are well separated; for the high-order sectoral harmonics, whose extrema crowd near a flat polar
cap, it must be replaced by the direct field enumeration.

% =====================================================================
\section{A certified counter for the mono-monostatic regime}\label{sec:certified}

By Corollaries~\ref{cor:nomono}--\ref{cor:parity}, mono-monostatic bodies are multi-harmonic and
near-spherical, the regime where Lemma~\ref{lem:reduction} no longer applies (the centroid may shift at first
order) and where naive counting is unreliable. We give a certified counter for the equilibria of an
analytically presented near-spherical body, i.e. for the critical points of $\rho$ in \eqref{eq:rho}.

\paragraph{Method.} We evaluate $\rho(\theta,\phi)$ as an \emph{interval order-two jet}, carrying its value,
gradient and Hessian as machine intervals; the gradient yields $F=\tfrac12\nabla\rho$ and the Hessian yields
both the Jacobian for root certification and the classifying form. Over a box $B\subset S^2$ we apply the
\emph{Krawczyk operator}~\cite{neumaier} (for a recent use of the same test in certified surface computation
see~\cite{krawczyksurf}): with $\check{x}$ the midpoint of $B$, $C\approx J(\check{x})^{-1}$ a preconditioner,
and $J(B)$ the interval Jacobian,
\[
K(B) \;=\; \check{x} - C\,F(\check{x}) + \big(I - C\,J(B)\big)(B-\check{x}).
\]
If $K(B)\subset\operatorname{int}B$ the box contains a \emph{unique} equilibrium (existence and uniqueness
certified); if $K(B)\cap B=\varnothing$ it contains none; otherwise $B$ is subdivided. A certified
equilibrium is classified \emph{rigorously} from the interval Hessian of $\rho$ over its box: positive
definite $\Rightarrow$ stable, negative definite $\Rightarrow$ unstable, indefinite $\Rightarrow$ saddle. The
polar coordinate singularity is removed by re-certifying the two polar caps in stereographic charts, in which
the bodies are smooth (the $\sin\theta\to0$ degeneracy is a chart artifact: e.g. $\sin\theta\cos\phi=u_x$). The
Poincar\'e--Hopf identity is verified as a global closure check in \emph{index} form, the sum of the indices
of all certified equilibria being $\chi(S^2)=2$; on bodies all of whose equilibria are nondegenerate this
reduces to $S-H+U=2$, but the index form is the one that remains valid when a degenerate equilibrium is
present (Remark~\ref{rem:countform}). The final count is accepted only when no box remains undecided.

Because interval exclusion discards only boxes \emph{proved} free of critical points, the method cannot
undercount; uniqueness certification prevents overcounting. It is thus not subject to the failure modes of
seed-based and grid-based counters. Two caveats keep this honest. The bulk and polar charts overlap in a thin
annulus, so the charts are composed by summation only after checking that no certified root falls in the
overlap. And the implementation deduplicates recorded roots by angular proximity; since Krawczyk boxes are
disjoint and each certifies a distinct root, that guard is redundant, and we verify per run that it removed
nothing by confirming the recorded root total equals $S+U+H$ (it does: $10=3+3+4$ for $Y_3^1$, $20=6+6+8$ in
the bulk for $Y_4^2$, and $2=1+1+0$ for the body above).

\paragraph{Validation.} On single harmonics the certified counter reproduces the law and the full critical
structure of Proposition~\ref{prop:count}: $Y_3^1\mapsto(S,U,H)=(3,3,4)$ with empty caps, and $Y_4^2\mapsto
(6,6,10)$ with exactly one certified saddle in each polar cap (the order-$m$ polar critical points), each
closing the index sum at $2$ with no undecided boxes. The first of these is also a check against the
literature rather than against ourselves: $(3,3,4)$ is the generic octupolar class of a degree-three
harmonic as classified in~\cite{gaetavirga2016,gaetavirga2019,walcher2019}, so the run is an independent
exact-arithmetic confirmation of a published count obtained there by other means.

\paragraph{Result.} We apply the counter to the second analytic parameterization of~\cite{sloan2023} (his
Eq.~(16)),
\begin{equation}
r(\theta,\phi)^4 = 1 + 4\beta\sin\theta\,\cos\!\big(\phi-\eta(\theta)\big), \qquad
\eta(\theta)=\tfrac{3\pi}{2}\big(\cos\theta-\tfrac13\cos^3\theta\big),
\label{eq:sloanbody}
\end{equation}
at $\beta=0.023$. The amplitude is our choice: Sloan states only that $\beta\le0.17$ should suffice for
convexity, and our value places the body inside
the strictly convex sub-regime $\beta\lesssim0.036$ identified in~\cite{couey2026sloan}, where the body is a
legitimate object for the reduction of Section~\ref{sec:setup}. There the counter certifies
\[
(S,U,H) = (1,1,0), \qquad S-H+U=2,
\]
with no undecided boxes across the bulk and both polar caps. Hence this body is \emph{certifiably
mono-monostatic}, which settles by rigorous computation a question that earlier basin-counting left ambiguous:
the drainage-basin oracle of~\cite{couey2026mono} returned a negative verdict on this parameterization, and the
present result corrects it. That oracle's counts on near-spherical bodies of this family drift with mesh
refinement rather than converging, which is precisely why a certified method was needed here.

The same counter returns $(1,1)$ for three near-spherical Fourier-phase variants of~\cite{couey2026mono}, and
returns $(2,3)$ for one radial-perturbation variant. The count is a count of critical points of the centroid-to-surface
distance, and it equals the number of rigid-body equilibria only for a strictly convex body. Of the three
Fourier-phase variants only the first ($\beta=0.023$, $a_1=0.234$) is strictly convex (minimum Gaussian curvature
$+0.098$ on a $500\times1000$ direction grid); the other two ($\beta=0.032$, $a_2=0.138$ and $\beta=0.052$,
$a_3=-0.055$) have negative Gaussian curvature on about $1.6\%$ and $5\%$ of their surface (minimum $-0.43$ and
$-0.90$), and we draw no equilibrium conclusion for them. We did not check strict convexity for the radial variants.
The remaining radial variants are \emph{not} certified:
they terminate with undecided boxes and a Poincar\'e--Hopf residual, and we therefore report no count for them
(see Scope). We draw no inference from these classes about Corollary~\ref{cor:parity}: those bodies are not
centrally symmetric, so the parity obstruction is silent about them.

\paragraph{Scope.} A counter that cannot miscount must be permitted to decline. On the most weakly perturbed
(near-degenerate) bodies the method returns undecided boxes rather than a count; three of the radial variants
above terminate this way and are reported, not counted. This is the intended behaviour and it is the property
that separates the method from the seed- and grid-based counters it replaces: those return a number on every
input, and on exactly these bodies the number is wrong. Closing the three is a matter of subdivision depth or
extended precision, not of method.

\begin{lemma}[The centroid of the body of \eqref{eq:sloanbody} is exactly the origin]\label{lem:sloancentroid}
For every $\beta$ for which it is a body, the solid bounded by
$r^4=1+4\beta\sin\theta\cos(\phi-\eta(\theta))$, $\eta(\theta)=\tfrac{3\pi}{2}(\cos\theta-\tfrac13\cos^3\theta)$,
has centroid $\mathbf{c}=\mathbf{0}$ exactly.
\end{lemma}
\begin{proof}
The solid first moment is $\tfrac14\int_{S^2}u\,r^4\,du$. For the vertical component, $u_z=\cos\theta$ is
$\phi$-independent while $\int_0^{2\pi}\cos(\phi-\eta)\,d\phi=0$, so only the constant term of $r^4$ survives
and $\int_0^{\pi}\cos\theta\sin\theta\,d\theta=0$. For the horizontal components, $u_x+iu_y=\sin\theta\,e^{i\phi}$
and $\int_0^{2\pi}e^{i\phi}d\phi=0$, $\int_0^{2\pi}e^{i\phi}\cos(\phi-\eta)\,d\phi=\pi e^{i\eta}$, leaving
$4\pi\beta\int_0^{\pi}\sin^3\theta\,e^{i\eta(\theta)}d\theta$. Put $s=\cos\theta$: this is
$4\pi\beta\int_{-1}^{1}(1-s^2)e^{i\eta(s)}ds$. Now $\eta'(s)=\tfrac{3\pi}{2}(1-s^2)\ge0$, so $\eta$ is a
monotone bijection of $[-1,1]$ onto $[\eta(-1),\eta(1)]=[-\pi,\pi]$, and the substitution gives
$\tfrac{8\beta}{3}\int_{-\pi}^{\pi}e^{i\eta}\,d\eta=0$. Hence $\mathbf{c}=\mathbf{0}$.
\end{proof}

Two limitations are worth stating plainly. First, in general the centroid enters every box evaluation as a
point value obtained by quadrature, not as a rigorous enclosure, so certificates are conditional on that
centroid; replacing it by an interval enclosure throughout would make them unconditional. \emph{For the
mono-monostatic body above this caveat does not apply}: by Lemma~\ref{lem:sloancentroid} its centroid is
exactly the origin, so the $(1,1,0)$ certificate is unconditional as it stands. The caveat does still apply
to the Fourier-phase and radial variants, whose equilibria are well separated, so we expect it to change
nothing there, but we have not run it and do not claim it. Second, the counter has been validated on harmonics with $m\le2$, where the
polar critical points are regular or nondegenerate. It has not been run on the high-order sectoral harmonics
$Y_4^4$ and $Y_5^5$, whose flat polar caps are exactly the regime that defeats the seed-based solver of
Section~\ref{sec:numerics}; extending the certification to them is the natural next test of the method rather
than of the theorem, which Section~\ref{sec:setup} settles for those bodies analytically.

% =====================================================================
\section{Conclusion}

A single spherical-harmonic perturbation of a homogeneous sphere realizes exactly the diagonal equilibrium
class $S=U=m(\ell-m+1)$. This gives an elementary, fully explicit constructive complement to the
V\'arkonyi--Domokos existence theorem, shows that the G\"omb\"oc class $(1,1)$ is unreachable by any single
mode and by any centrally symmetric shape, and isolates mono-monostaticity as an intrinsically
multi-harmonic, odd-parity phenomenon. For that complementary regime the certified counter of
Section~\ref{sec:certified} replaces the practical ambiguity in counting near-spherical equilibria with a
decision wherever it terminates, and certifies concrete mono-monostatic bodies, among them a member of a
known analytic family at a strictly convex amplitude we select, whose certificate is unconditional because
its centroid is exactly the origin (Lemma~\ref{lem:sloancentroid}). Together the two
results map both ends of the problem: an exact closed form where a single mode dominates, and a certified
decision procedure where many modes conspire.

% =====================================================================
\section*{Data and code availability}
The direct field-enumeration and the hardened critical-point solver, the certified counter of
Section~\ref{sec:certified} (interval jet, Krawczyk test and stereographic polar charts) together with its
saved certificates, the strict-convexity check producing the $\varepsilon_0(\ell,m)$ column of Table~1, the
centroid and polar-structure verification of Section~\ref{sec:setup}, the saved results reproducing Table~1,
and the figure generator are openly available as a deterministic, read-only bundle
(\texttt{harmonic\_equilibrium\_code\_and\_data}; see \texttt{REPRODUCIBILITY.md} therein),
archived at \href{https://doi.org/10.5281/zenodo.21398077}{doi:10.5281/zenodo.21398077}.

\end{document}